\documentclass[12pt,onecolumn,fleqn]{IEEEtran}

\usepackage{graphicx}

\ifCLASSINFOpdf
\else
\fi
\usepackage{amsmath}
\usepackage{amsthm}
\usepackage{amsfonts}
\usepackage{amssymb}
\usepackage{color}

\newtheorem{theorem}{Theorem}
\newtheorem{lemma}[theorem]{Lemma}

\newtheorem{remark}[theorem]{Remark}

\begin{document}
%
\title{Towards a minimal order distributed observer for linear systems}
%
%
%


\author{Weixin~Han,
        Harry~L.~Trentelman,
        ~Zhenhua~Wang,
        and~Yi~Shen
\thanks{This work was partially supported by China Scholarship Council and National Natural Science Foundation of China (Grant No. 61273162, 61403104).}
\thanks{Weixin Han, Zhenhua Wang and Yi Shen are with the Department of Control Science and Engineering, Harbin Institute of Technology, Harbin,
150001 P. R. China.
        {\tt\small zhenhua.wang@hit.edu.cn}}%
\thanks{Harry L. Trentelman is with the Johann Bernoulli Institute for Mathematics and Computer Science, University of Groningen,
9700 AK Groningen The Netherlands.
        {\tt\small h.l.trentelman@rug.nl}}%
}

%
%

\markboth{Journal of \LaTeX\ Class Files,~Vol.~XX, No.~X, August~2017}%
{Shell \MakeLowercase{\textit{et al.}}: Bare Demo of IEEEtran.cls for IEEE Journals}
%



\maketitle


\begin{abstract}
In this paper we consider the distributed estimation problem for continuous-time linear time-invariant (LTI) systems. A single linear plant is observed by a network of local observers. Each local observer in the network has access to only part of the output of the observed system, but can also receive information on the state estimates of its neigbours. Each local observer should in this way generate an estimate of the plant state. In this paper we study the problem of existence of a reduced order distributed observer. We show that if the observed system is observable and the network graph is a strongly connected directed graph, then a distributed observer exists with state space dimension equal to $Nn - \sum_{i =1}^N p_i$, where $N$ is the number of network nodes, $n$ is the state space dimension of the observed plant, and $p_i$ is the rank of the output matrix of the observed output received by the $i$th local observer. In the case of a single observer, this result specializes to the well-known minimal order observer in classical observer design.
\end{abstract}


%
\IEEEpeerreviewmaketitle

\section{Introduction}
%
%
%
%

Recently, there has been much interest in the problem of designing distributed observers for estimation of the state of a given linear time invariant plant. Whereas the classical observer problem is to find a single observer that receives the entire measured plant output in order to generate this state estimate, in the distributed version the aim is to find a given number of local observers that
can communicate according to an a priori given network graph. Each of the local observers in the network receives only part of the plant output, but also information on the state estimates of its neigbours. Each local observer should in this way generate an estimate of the plant state. Thus, the problem of finding a distributed observer can be interpreted as the problem of finding a single observer that consists of a given number of local observers, interconnected by means of an a priori given network graph. Since each of the local observers receives only part of the plant output, properties like observability or detectability that might hold for the original plant output do no longer hold for the partial output, and hence classical observer design is not applicable for the local observer.

Among the many contributions on the distributed observer problem we mention  \cite{Kim2016CDC}, \cite{Mitra2016CCC} and \cite{Park2012ACC}. 
In particular, in \cite{Park2012ACC,Park2012CDC,Park2017TAC} a state augmented observer was constructed to cast the distributed estimation problem as a problem of decentralized stabilization, using the notion of fixed modes \cite{Anderson1981AUT}. These references only discuss discrete-time systems. More recently, in \cite{Wang2017TAC}, the idea of putting the distributed observer problem in the context of decentralized control was applied to continuous time plants. In \cite{Mitra2016CCC,Mitra2016CDC,Mitra2017} local Luenberger observers at each node were constructed, based on applying the Kalman observable decomposition. There, the observer reconstructs a certain portion of the state solely by using its own measurements, and uses consensus dynamics to estimate the unobservable portions of the state at each node. Specifically, in \cite{Kim2016CDC} two observer gains were designed to achieve distributed state estimation, one for local measurements and the other for the information exchange. In \cite{Han2017}, a simple LMI based approach was proposed for the design of distributed observers. 

A standard result in classical observer design states that if the plant is observable, then an observer with arbitrary fast error convergence exists of order equal to the order of the plant, say $n$,  minus the rank of the output matrix, say $p$, \cite{Luenberger1971TAC}. It was argued in \cite{Wonham} that indeed $n - p$ is the minimal order for state observers. Of course, similarly one can address the issue of existence of a reduced, or even minimal, order distributed observer. This issue will be the topic of the present paper. Assume that our plant is a continuous-time LTI system
\begin{equation} \label{sys0}
\begin{array}{l}
\dot{x}=Ax\\
y=Cx
\end{array}
\end{equation}
where $x\in \mathbb{R}^n$ is the state and $y\in \mathbb{R}^m$ is the measurement output. We partition the output $y$ as 
\[
y = {\footnotesize \begin{bmatrix} y_1 \\ y_2 \\ \vdots \\ y_N \end{bmatrix} }
\]
where $y_i\in \mathbb{R}^{m_i}$ and $\sum_{i=1}^Nm_i=m$. Accordingly, we partition the output matrix as 
\[
C = {\footnotesize \begin{bmatrix} C_1 \\ C_2 \\ \vdots \\ C_N \end{bmatrix} }
\]
with $C_i\in \mathbb{R}^{m_i\times n}$. In addition, a directed graph with $N$ nodes is given. Each node in the graph will carry a local observer. The local observer at node $i$ has only access to the measurement $y_i=C_ix$ and to the state estimates of its neighbours, including itself. In this paper, a standing assumption will be that the communication graph is strongly connected. We will also assume that the pair $(C,A)$ is observable. For the discrete time case, it was shown in \cite{Park2017TAC} that a distributed observer of order $Nn + N -1$ exists. This bound was re-established in  \cite{Wang2017TAC} for continuous time plants. Again for the discrete time case, in \cite{Mitra2017} 
it was shown that a distributed observer exists of order $Nn$. Also in  \cite{Kim2016CDC}, under certain assumptions, a dynamic order $Nn$ was shown to be sufficient. More recently, in our paper \cite{Han2017} we reconfirmed that for the continuous time case a dynamic order $Nn$ suffices. 

In the present paper we will improve all sufficient dynamic orders established up to now and as our main result  show that, for any desired errror convergence rate, a distributed observer exists of dynamic order equal to $Nn - \sum_{i=1}^N p_i$, where  $p_i$ is the rank of the local output matrix $C_i$. This result extends in a natural way the minimal order $n -p$ for a single, non-distributed observer, with $p$ the rank of the output matrix $C$.

\begin{figure}
  \centering
  \includegraphics[scale=0.45]{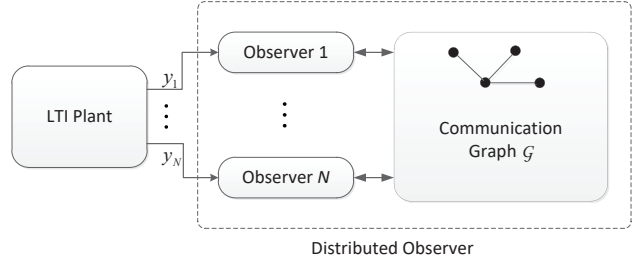}
  \caption{Framework for distributed state estimation}
  \label{dobs}
\end{figure}

%
%

\section{Preliminaries and Problem Formulation}
\subsection{Preliminaries}
\noindent \emph{Notation}:
 The rank of a given matrix $M$ is denoted by $\mathrm{rank}~M$. If $M$ has full column rank $m$ then $M^{\dag}=(M^TM)^{-1}M^T$ denotes its Moore-Penrose inverse, so $M^{\dag}M=I_m$. The identity matrix of dimension $N$ will be denoted by $I_N$. The vector $\mathbf{1}_N$ denotes the $N$-dimensional column vector comprising of all ones. 
For a symmetric matrix $P$, $P>0$ $(P<0)$ means that $P$ is positive (negative) definite. For a set $\{A_1,A_2,\cdots,A_N\}$ of matrices, we use $\mathrm{diag}\{A_1,A_2,\cdots,A_N\}$ to denote the block diagonal matrix with the $A_i$'s along the diagonal, and the matrix $\begin{bmatrix}
A_1^T &A_2^T&\cdots &A_N^T
\end{bmatrix}^T $ is denoted by $\mathrm{col}(A_1,A_2,\cdots,A_N)$. The Kronecker product of the matrices $M_1$ and $M_2$ is denoted by $M_1\otimes M_2$. In this paper, $\mathbb{R}^n$ will denote the $n$-dimensional Euclidean space. For a $p \times n$ matrix $A$, $\mathrm{ker}~A:=\{x\in \mathbb{R}^n \mid Ax=0\}$ and $\mathrm{im}~A:= \{ Ax \mid x \in \mathbb{R}^n\}$ will denote the kernel and image of $A$, respectively. If $\mathcal{V}$ is a subspace of $\mathbb{R}^n$, then $\mathcal{V}^{\perp}$ will denote the orthogonal complement of $\mathcal{V}$ with respect to the standard inner product in $\mathbb{R}^n$.

In this paper, a weighted directed graph is denoted by $\mathcal{G=(N,E,A)}$, where $\mathcal{N} = \{1,2,\cdots,N\}$ is a finite nonempty set of nodes, $\mathcal{E\subset N\times N}$ is an edge set of ordered pairs of nodes, and $\mathcal{A}=[a_{ij}]\in \mathbb{R}^{N\times N}$ denotes the adjacency matrix. The $(j,i)$-th entry $a_{ji}$ is the weight associated with the edge $(i,j)$. We have $a_{ji}\neq 0$ if and only if $(i,j)\in \mathcal{E}$. Otherwise $a_{ji}=0$. An edge $(i,j)\in \mathcal{E}$ designates that the information flows from node $i$ to node $j$. 
A directed path from node $i_1$ to $i_l$ is a sequence of edges $(i_k, i_{k+1})$, $k=1,2,\cdots,l-1$ in the graph. A directed graph $\mathcal{G}$ is strongly connected if between any pair of distinct nodes $i$ and $j$ in $\mathcal{G}$, there exists a directed path from $i$ to $j$, $i,j\in \mathcal{N}$.

The Laplacian $\mathcal{L}=[l_{ij}]\in \mathbb{R}^{N\times N}$ of $\mathcal{G}$ is defined as $\mathcal{L:=D-A}$, where the $i$-th diagonal entry of the diagonal matrix $\mathcal{D}$ is given by $d_i=\sum_{j=1}^N a_{ij}$. By construction, $\mathcal{L}$ has a zero eigenvalue with a corresponding eigenvector $\mathbf{1}_N$ (i.e., $\mathcal{L}\mathbf{1}_N=0_N$), and if the graph is strongly connected, its algebraic multipicity is equal to one and all the other eigenvalues lie in the open right-half complex plane. 

For strongly connected graphs $\mathcal{G}$, we now review the following lemma.
\begin{lemma}\cite{Olfati2004TAC,Ren2005TAC,Yu2010E3T} \label{mirror}
Assume $\mathcal{G}$ is a strongly connected directed graph. Then there exists a unique positive row vector $r=\begin{bmatrix}
r_1,\cdots,r_N
\end{bmatrix} $ such that $r\mathcal{L}=0$ and $r\mathbf{1}_N=N$. Define $R:=\mathrm{diag}\{r_1,\cdots,r_N\}$. Then $\hat{\mathcal{L}}:=R\mathcal{L+L}^TR$ is positive semi-definite, $\mathbf{1}_N^T\hat{\mathcal{L}}=0$ and $\hat{\mathcal{L}}\mathbf{1}_N=0$.
\end{lemma}

We note that $R\mathcal{L}$ is the Laplacian of the balanced directed graph obtained by adjusting the weights in the original graph. The matrix $\hat{\mathcal{L}}$ is the Laplacian of the undirected graph obtained by taking the union of the edges and their reversed edges in this balanced digraph. This undirected graph is called the mirror of this balanced graph \cite{Olfati2004TAC}.
%
%

\subsection{Problem formulation and main result}
Consider the continuous-time LTI system \eqref{sys0}, where $x\in \mathbb{R}^n$ is the state and $y\in \mathbb{R}^m$ is the measurement output. As explained in the introduction we partition the output $y$ as $y=\mathrm{col}(y_1,\cdots,y_N)$, where $y_i\in \mathbb{R}^{m_i}$ and $\sum_{i=1}^Nm_i=m$. Accordingly, $C=\mathrm{col}(C_1,\cdots,C_N)$ with $C_i\in \mathbb{R}^{m_i\times n}$. Here, the portion $y_i=C_ix$ is assumed to be the only output information that can be acquired by node $i$ in the given network graph $\mathcal{G}$. The rank of the local output matrix $C_i$ will be denoted by  $p_i$.

In this paper, a standing assumption will be that the communication graph $\mathcal{G}$ is a strongly connected directed graph. We will also assume that the pair $(C,A)$ is observable. However, $(C_i,A)$ is not assumed to be observable or detectable.

We will design a  distributed observer for the system \eqref{sys0} with the given communication network  $\mathcal{G}$. The distributed observer will consist of $N$ local observers, and the local observer at node $i$ will have dynamics of the following form:
\begin{equation} \label{obsl}
\begin{array}{l}
\dot{z}_i=N_iz_i+L_iy_i+\gamma r_iM_i\sum_{j = 1}^N a_{ij}(\hat{x}_j-\hat{x}_i)\\
\hat{x}_i=P_iz_i+Q_iy_i
\end{array}
\end{equation}
where $i \in \mathcal{N}$, $z_i\in \mathbb{R}^{n-p_i}$ is the state of the local observer, $\hat{x}_i\in \mathbb{R}^n$ is the estimate of plant state at node $i$, $a _{ij}$ is the $(i,j)$-th entry of the adjacency matrix $\mathcal{A}$ of the given network, $r_i$ is defined as in Lemma 1, $\gamma \in \mathbb{R}$ is a coupling gain to be designed, $N_i \in \mathbb{R}^{(n-p_i) \times (n-p_i)}$, $L_i \in \mathbb{R}^{(n-p_i) \times m_i}$, $M_i \in \mathbb{R}^{(n-p_i)\times n}$, $P_i \in \mathbb{R}^{n \times (n - p_i)}$ and $Q_i \in \mathbb{R}^{n \times m_i}$ are gain matrices to be designed.
%
The objective of distributed state estimation is to design a network of local observers \eqref{obsl} that cooperatively estimate the state of the plant \eqref{sys0}. In other words, we want to design \eqref{obsl} such that for any choice of initial states on \eqref{sys0} and \eqref{obsl} 
\begin{equation}\label{esg}
\lim_{t\to \infty} \left( \hat{x}_i(t)-x(t) \right) =0
\end{equation}
for all $i\in\mathcal{N}$, i.e., the state estimate maintained by each node converges to the true state of the plant. Following \cite{Park2017TAC}, if the distributed observer \eqref{obsl} achieves \eqref{esg} then it is said to achieve omniscience asymptotically.
The main result of this paper is the following:
\begin{theorem}
Assume that $(C,A)$ is observable and that the network graph $\mathcal{G}$ is a strongly connected directed graph. Let $\alpha \geq 0$ be a desired error convergence rate. Then there exists a distributed observer \eqref{obsl} that achieves omniscience asymptotically and all error trajectories  converge to zero with convergence rate at least $\alpha$. Such distributed observer exists with state space dimension $Nn - \sum_{i =1}^N p_i$, where $p_i := \mathrm{rank}~C_i$.
\end{theorem}
In the remainder of this paper we will prove this result by outlining how to design a desired distributed observer. 

\section{Design of the distributed observer}

To design a distributed observer of the form \eqref{obsl}, we make a full rank factorization for each local output matrix $C_i$. Recall that $\mathrm{rank}~C_i =p_i$ and factorize $C_i=D_iF_i$ with $D_i \in \mathbb{R}^{m_i\times p_i}$ full column rank and $F_i \in \mathbb{R}^{p_i\times n}$ full row rank. 
Recall that $D_i^{\dag}=(D_i^TD_i)^{-1}D_i^T$. Since $y_i=C_ix$, we have 
\begin{equation}
\tilde{y}_i:=D_i^{\dag}y_i=F_ix,
\end{equation}
where $\tilde{y}_i \in \mathbb{R}^{p_i}$ represents a virtual local output. 
Denote 
$F = \mathrm{col}(F_1,\cdots,F_N)$. Clearly, $(F,A)$ is observable, but for $i \in \mathcal{N}$, $(F_i,A)$ is not necessarily observable or detectable. 
%
To proceed, we introduce orthogonal transformations that yields observability decompositions for the pairs $(F_i,A)$. For $i \in \mathcal{N}$, let $T_i$ be an orthogonal matrix such that the matrices $A$ and $F_i$ are transformed by the state space transformation $T_i$ into the form 
\begin{equation} \label{trf}
T_i^TAT_i=\begin{bmatrix}
A_{i11} &A_{i12} &0\\
A_{i21} &A_{i22} &0\\
A_{i31} &A_{i32} &A_{iu}\\
\end{bmatrix},~F_iT_i=\begin{bmatrix}
E_i &0 &0
\end{bmatrix}
\end{equation}
where $A_{i11}\in \mathbb{R}^{p_i\times p_i}$, $A_{i12}\in \mathbb{R}^{p_i\times (v_i-p_i)}$, $A_{i21}\in \mathbb{R}^{(v_i-p_i)\times p_i}$, $A_{i22}\in \mathbb{R}^{(v_i-p_i)\times (v_i-p_i)}$, $A_{i31}\in \mathbb{R}^{(n-v_i)\times p_i}$, $A_{i32}\in \mathbb{R}^{(n-v_i)\times (v_i-p_i)}$, $A_{iu}\in \mathbb{R}^{(n-v_i)\times (n-v_i)}$, $E_i\in \mathbb{R}^{p_i\times p_i}$ is a non-singular matrix, and $n-v_i$ is the dimension of the unobservable subspace of the pair $(F_i,A)$.

For convenience, denote 
\begin{equation}
A_{io}=\begin{bmatrix}
A_{i11} &A_{i12}\\
A_{i21} &A_{i22} 
\end{bmatrix}, A_{ir}=\begin{bmatrix}
A_{i31} &A_{i32}
\end{bmatrix}, F_{io}=\begin{bmatrix}
E_i &0
\end{bmatrix},
\end{equation}
where $A_{io}\in \mathbb{R}^{v_i\times v_i}$, $A_{ir}\in \mathbb{R}^{(n-v_i)\times v_i}$, $F_{io}\in \mathbb{R}^{p_i\times v_i}$. Then clearly
\begin{equation} \label{trfc}
T_i^TAT_i=\begin{bmatrix}
A_{io} &0\\
A_{ir} &A_{iu}
\end{bmatrix},~F_iT_i=\begin{bmatrix}
F_{io} &0
\end{bmatrix}.
\end{equation}
By construction, the pair $(F_{io},A_{io})$ is observable. Furthermore, it can be checked using the Hautus test that the pair $(A_{i12},A_{i22})$ is also observable. Since $E_i$ is nonsingular, then also the pair $(E_iA_{i12},A_{i22})$ is observable.

In addition, if we partition $T_i=\begin{bmatrix}
T_{i1} &T_{i2}
\end{bmatrix} $, where $T_{i1}$ consists of the first $v_i$ columns of $T_i$, then the unobservable subspace is given by $\mathrm{im}~T_{i2}=\mathrm{ker}~O_{Fi}$,
where $O_{Fi}=\mathrm{col}(F_i,F_iA,\cdots,F_iA^{n-1})$. Note that $\mathrm{im}~T_{i1}=(\mathrm{ker}~O_{Fi})^{\perp}$. 

We now proceed with defining the gain matrices $P_i$ and $Q_i$ in the output equation of 
\eqref{obsl}. For $i \in \mathcal{N}$, define $S_i \in \mathbb{R}^{n \times (n - p_i)}$ and $K_i \in \mathbb{R}^{n \times m_i}$ by
\begin{equation}  \label{outgain1}
S_i := \begin{bmatrix} 0 \\ I_{n - p_i} \end{bmatrix} ~\mbox{ and }~ K_i := \begin{bmatrix} E_i^{-1} \\ H_i \\ 0 \end{bmatrix} D_i^{\dagger}.
\end{equation}
Here, $H_i \in \mathbb{R}^{(v_i - p_i) \times p_i}$ still needs to be defined. Now define 
\begin{equation} \label {Tis}
T_{is} :=  T_i S_i
\end{equation}
as the $n \times (n -p_i)$ matrix consisting of the last $n-p_i$ columns of the orthogonal matrix $T_i$. Next define
\begin{equation}  \label{outgain2}
P_i := T_{is} ~\mbox{ and }~ Q_i := T_i K_i
\end{equation}
To analyze and further synthesize the local observer (\ref{obsl}), we define the local estimation error of the $i$-th observer as
\begin{equation} \label{err}
e_i:=\hat{x}_i-x.
\end{equation}
Using the definitions \eqref{outgain2} and combining (\ref{sys0}) and (\ref{obsl}) shows that $e_i$ satisfies:
\begin{equation} \label{errordynamics}
\begin{array}{ll}
\dot{e}_i & = P_i \dot{z}_i + Q_i \dot{y}_i - \dot{x}\\
&= T_{is} \dot{z}_i+T_i K_i\dot{y}_i -\dot{x}\\
&= T_iS_i \left( N_iz_i+L_iy_i+\gamma r_iM_i\sum_{j = 1}^N a_{ij}(\hat{x}_j-\hat{x}_i) \right)+(T_iK_iC_i-I)Ax\\
&= T_iS_i \left( N_iS_i^T (T_i^Te_i-K_iy_i+T_i^Tx)+L_iy_i+\gamma r_iM_i\sum_{j = 1}^N a_{ij}(\hat{x}_j-\hat{x}_i) \right)+(T_iK_iC_i-I)Ax\\
&= T_iS_iN_iS_i^T T_i^Te_i+\gamma r_iT_iS_iM_i\sum_{j = 1}^N a_{ij}(e_j-e_i)\\
&\quad+T_i \left( (S_iL_i-S_iN_iS_i^T K_i)D_iF_iT_i+S_iN_iS_i^T+(K_iD_iF_iT_i-I)T_i^TAT_i \right)T_i^Tx
\end{array}
\end{equation}

As a first step to achieve stable error dynamics it is required that the right hand side of the differential equation \eqref{errordynamics} does not depend on the state $x$. This can be achieved by choosing the local observer gain matrices $N_i$ and $L_i$ in such a way that 
\begin{equation}\label{eeq}
\begin{array}{l}
(S_iL_i-S_iN_iS_i^T K_i)D_iF_iT_i+S_iN_iS_i^T+(K_iD_iF_iT_i-I)T_i^TAT_i=0
\end{array}.
\end{equation}
It can be checked by straightforward verification that \eqref{eeq} is achieved by choosing \begin{equation} \label{out1}
N_i=\begin{bmatrix}
A_{i22}-H_iE_iA_{i12} &0\\
A_{i32} &A_{iu}
\end{bmatrix},
\end{equation}
\begin{equation}  \label{out2}
L_i=\begin{bmatrix}
A_{i21}-H_iE_iA_{i11}\\
A_{i31}
\end{bmatrix}E_i^{-1}D_i^{\dag}+N_iS_i^T K_i.
\end{equation}
Here, again we note that $H_i \in \mathbb{R}^{(v_i - p_i) \times p_i}$ still needs to be defined. With this choice of $N_i$ and $L_i$, the local error satisfies the differential equation
\begin{equation}
\dot{e}_i= T_{is}N_i T_{is}^T e_i+\gamma r_iT_{is} M_i\sum_{j = 1}^N a_{ij}(e_j-e_i).
\end{equation}
Let $e:=\mathrm{col}(e_1,e_2,\cdots,e_N)$ be the joint vector of errors. 
Define 
\begin{equation} \label{Ts}
T_s :=\mathrm{diag}\{T_{1s},\cdots,T_{Ns}\},
\end{equation}
\begin{equation} \label {M}
M=\mathrm{diag}\{M_1,\cdots,M_N\},
\end{equation}
\begin{equation} \label {N}
N=\mathrm{diag}\{N_1,\cdots,N_N\}.
\end{equation}
Clearly then, each global error trajectory satisfies the differential equation
\begin{equation} \label{sys_eg}
\dot{e}= \left( T_{s}N T_{s}^T -\gamma T_s M(R\mathcal{L}\otimes I_n) \right) e,
\end{equation}
where $R$ is as defined in Lemma 1.  

Note that $\mathrm{im}~T_s$ is an invariant subspace for the differential equation \eqref{sys_eg}.
Even more, it can be shown that each feasible global error trajectory $e$ lives in the subspace $\mathrm{im}~T_s$. We state this as a lemma:
\begin{lemma} \label{invariance}
Assume that the gain matrices $P_i$, $Q_i$, $N_i$ and $L_i$ are given by 
\eqref{outgain2}, \eqref{outgain1}, \eqref{out1} and \eqref{out2}. Let $e:=\mathrm{col}(e_1,e_2,\cdots,e_N)$ be the joint vector of errors, with for $i \in \mathcal{N}$ the local error equal to $e_i = \hat{x}_i  - x$, where $x$ is a trajectory of the plant \eqref{sys0} and $\hat{x}_i$ satisfies \eqref{obsl}. Then $e(t) \in \mathrm{im}~T_s$ for all $t \in \mathbb{R}$.
\end{lemma}
\begin{proof}
For $i \in \mathcal{N}$, let $T_{ip}$  be the $n \times p_i$ matrix consisting of the first $p_i$ columns of $T_i$. Then we have $T_i = \begin{bmatrix} T_{ip} & T_{is} \end{bmatrix}$. Since $T_i$ is orthogonal, we have $\mathrm{im}~ T_{is} = \mathrm{ker}~ T_{ip}^T$. Let $e_i$ be a local error trajectory. We have
\begin{equation}\label{xhe}
\begin{array}{ll}
T_{ip}^Te_i&= T_{ip}^T(\hat{x}_i-x)\\
&=T_{ip}^TT_{is} z_i +T_{ip}^TT_iK_iC_ix-T_{ip}^Tx\\
&= \begin{bmatrix}
I_{p_i} &0
\end{bmatrix}K_iC_ix-T_{ip}^Tx\\
&=(E_i^{-1}F_i-T_{ip}^T)x.
\end{array}
\end{equation}
By \eqref{trf} we have
\begin{equation}
F_i\begin{bmatrix}
T_{ip} &T_{is}
\end{bmatrix}=\begin{bmatrix}
E_i &0
\end{bmatrix},
\end{equation}
which implies
\begin{equation}
F_i=\begin{bmatrix}
E_i &0
\end{bmatrix}\begin{bmatrix}
T_{ip} &T_{is}
\end{bmatrix}^T=E_iT_{ip}^T.
\end{equation}
Thus we obtain $T_{ip}^Te_i =0$ and hence $e_i(t) \in  \mathrm{ker}~ T_{ip}^T$ for all $t \in \mathbb{R}$. We conclude that $e_i(t) \in \mathrm{im}~ T_{is}$ so $e(t) \in \mathrm{im}~ T_{s}$  for all $t \in \mathbb{R}$.
\end{proof} 
From Lemma \ref{invariance} we infer that the distributed observer \eqref{obsl} achieves omniscience asymptotically (\ref{esg}) if each solution $e$ of \eqref{sys_eg} such that $e(t) \in \mathrm{im}~T_s$ for all $t \in \mathbb{R}$ converges to zero as $t$ runs off to infinity. 

Up to now, we have specified in the to be designed local observer \eqref{obsl} the gain matrices $P_i$, $Q_i$, $N_i$ and $L_i$. However, $Q_i$, $N_i$ and $L_i$ still depend on the parameter matrix $H_i \in \mathbb{R}^{(v_i - p_i) \times p_i}$ that has to be specified. Also the matrix $M_i$ and coupling gain $\gamma$ still need to be specified. In order to proceed, we state the following two lemmas. The first of these is standard:
\begin{lemma}\cite{Li2014book} \label{standard}
For a strongly connected directed graph $\mathcal{G}$, zero is a simple eigenvalue of $\hat{\mathcal{L}}=R\mathcal{L+L}^TR$ introduced in Lemma 1. Furthermore, its eigenvalues can be ordered as $\lambda _1=0<\lambda _2\leqslant \lambda _3\leqslant \cdots \leqslant \lambda _N $. Furthermore, there exists an orthogonal matrix $U=\begin{bmatrix} \frac{1}{\sqrt{N}}\mathbf{1}_N &U_2 \end{bmatrix} $, where $U_2\in \mathbb{R}^{N \times (N-1)}$, such that $U^T(R\mathcal{L+L}^TR)U=\mathrm{diag}\{0,\lambda _2,\cdots,\lambda _N\}$.
\end{lemma}
Our second lemma was proven in \cite{Han2017}. In order to make this paper self contained, we also include the proof here.
\begin{lemma} \label{key}
Let $\mathcal{L}$ be the Laplacian matrix associated with the strongly connected directed graph $\mathcal{G}$. For all $g_i>0$, $i\in \mathcal{N}$, there exists $\epsilon>0$ such that
\begin{equation} \label{L1}
T^T((R\mathcal{L+L}^TR)\otimes I_n)T+G>\epsilon I_{nN},
\end{equation}
where $T=\mathrm{diag}\{T_1,\cdots,T_N\}$, $R$ is defined as in Lemma \ref{mirror}, $G=\mathrm{diag}\{G_1,\cdots,G_N\}$, and $G_i=\begin{bmatrix}
 g_i I_{v_i} & 0\\
 0 &0_{n-v_i}
\end{bmatrix}$, $i\in \mathcal{N}$.
\end{lemma}
\begin{proof}
The inequality (\ref{L1}) holds if and only if the following inequality holds.
\begin{equation} \label{L2}
(U^T(R\mathcal{L+L}^TR)U)\otimes I_n+(U^T\otimes I_n)TGT^T(U\otimes I_n)>0,
\end{equation}
where $U$ is as in Lemma \ref{standard}.
The inequality (\ref{L2}) holds if the following inequality holds:
\begin{equation} \label{L3}
\lambda _2 I_{Nn}-(U\otimes I_n)\begin{bmatrix}
\lambda _2 I_n &0\\
0 &0_{(N-1)n}
\end{bmatrix}(U^T\otimes I_n)+TGT^T>0.
\end{equation}
Since $U=\begin{bmatrix}
\frac{1}{\sqrt{N}}\mathbf{1}_N &U_2
\end{bmatrix} $ and $U^T=\begin{bmatrix}
\frac{1}{\sqrt{N}}\mathbf{1}_N^T \\
U_2^T
\end{bmatrix} $, the inequality (\ref{L3}) is equivalent to
\begin{equation} \label{L4}
\lambda _2 I_{Nn}-\frac{\lambda _2}{N}(\mathbf{1}_N\otimes I_n)(\mathbf{1}_N^T\otimes I_n)+TGT^T>0.
\end{equation}
By pre- and post- multiplying with $T^T$ and $T$, the inequality (\ref{L4}) is equivalent to
\begin{equation}
\lambda _2 I_{Nn}-\frac{\lambda _2}{N}T^T(\mathbf{1}_N\otimes I_n)(\mathbf{1}_N^T\otimes I_n)T+G>0,
\end{equation}
that is 
\begin{equation} \label{L5}
\lambda _2 I_{Nn}+G-\frac{\lambda _2}{N}\begin{bmatrix}
T_1 &\cdots &T_N
\end{bmatrix}^T\begin{bmatrix}
T_1 &\cdots &T_N
\end{bmatrix}>0.
\end{equation}
By using the Schur complement lemma \cite{Boyd1994}, the inequality (\ref{L5}) is equivalent to
\begin{equation} \label{L51}
\begin{bmatrix}
\Psi _1  &\cdots &0  &T_{1}^T\\
\vdots  &\ddots &\vdots&\vdots\\
0 &\cdots &\Psi _N &T_{N}^T\\
T_{1}  &\cdots  &T_{N}&\frac{N}{\lambda _2} I_n
\end{bmatrix}>0.
\end{equation}
where $\Psi _i :=\begin{bmatrix}
(\lambda _2+g_i)I_{v_i} &0 \\
0 &\lambda _2 I_{n-v_i}
\end{bmatrix}$, $i\in \mathcal{N}$.
Now, partition the orthogonal matrix $T_i$ as  $T_i=\begin{bmatrix}
T_{i1} &T_{i2}
\end{bmatrix}$, 
with $T_{i1} \in \mathbb{R}^{n \times v_i}$ and $T_{i2} \in \mathbb{R}^{n \times (n-v_i)}$, $i\in \mathcal{N}$. Clearly, $T_{i1}T_{i1}^T+T_{i2}T_{i2}^T=I_n$. Again using the Schur complement lemma, \eqref{L51} is then equivalent with
\begin{equation}\label{L6}
\begin{bmatrix}
\lambda _2 I_{n-v_1} &0 &\cdots&0  &T_{12}^T\\
0 &\Psi _2  &\cdots &0  &T_{2}^T\\
\vdots  &\vdots  &\ddots &\vdots&\vdots\\
0 &0 &\cdots &\Psi _N  &T_{N}^T\\
T_{12} &T_{2} &\cdots  &T_{N} &\frac{N}{\lambda _2} I_n-\frac{1}{\lambda _2 +g_1}T_{11}T_{11}^T
\end{bmatrix}>0.
\end{equation}
By repeatedly using the Schur complement lemma, we finally obtain that inequality (\ref{L5}) holds if and only if 
\begin{equation}\label{L8}
\frac{N}{\lambda _2} I_n-\sum_{i=1}^N \frac{1}{\lambda _2 +g_i}T_{i1}T_{i1}^T-\sum_{i=1}^N\frac{1}{\lambda _2 }T_{i2}T_{i2}^T>0.
\end{equation}
The left-hand side of inequality (\ref{L8}) is equal to
\begin{equation}
\begin{array}{ll}
&\frac{N}{\lambda _2} I_n-\sum_{i=1}^N \frac{1}{\lambda _2 +g_i}T_{i1}T_{i1}^T-\sum_{i=1}^N\frac{1}{\lambda _2 }T_{i2}T_{i2}^T\\
=&\frac{N}{\lambda _2} I_n-\sum_{i=1}^N\frac{1}{\lambda _2 }T_{i2}T_{i2}^T-\sum_{i=1}^N\frac{1}{\lambda _2 }T_{i1}T_{i1}^T\\
&+\sum_{i=1}^N\frac{1}{\lambda _2 }T_{i1}T_{i1}^T-\sum_{i=1}^N \frac{1}{\lambda _2 +g_i}T_{i1}T_{i1}^T\\
=&\sum_{i=1}^N (\frac{1}{\lambda _2}-\frac{1}{\lambda _2 +g_i})T_{i1}T_{i1}^T\\
\geqslant &\sum_{i=1}^N (\frac{1}{\lambda _2}-\frac{1}{\lambda _2 +g_{min}})T_{i1}T_{i1}^T\\
=&(\frac{N}{\lambda _2}-\frac{N}{\lambda _2 +g_{min}})\begin{bmatrix}
T_{11} &\cdots & T_{N1}
\end{bmatrix} \begin{bmatrix}
T_{11} &\cdots & T_{N1}
\end{bmatrix}^T,
\end{array}
\end{equation}
where $g_{min}$ is the minimum value of $g_i$, $i\in \mathcal{N}$. Obviously, we have $(\frac{N}{\lambda _2}-\frac{N }{\lambda _2 +g_{min}})>0 $ since $g_{min}>0$.

We will now prove that $\mathrm{rank}\begin{bmatrix}
T_{11} & T_{21} &\cdots & T_{N1}
\end{bmatrix}=n$, so that it has full row rank.
Indeed, for $T_{i1}$, we have
\begin{equation} \label{L9}
\mathrm{im}~T_{i1}=(\mathrm{ker}~O_{Fi})^{\perp}
\end{equation}
where $O_{Fi}=\mathrm{col}(F_i,F_iA,\cdots,F_iA^{n-1})$ is the unobservable subspace of $(F_i, A)$.
Hence,
\begin{equation}
\begin{array}{ll}
&(\mathrm{im}\begin{bmatrix}
T_{11} & T_{21} &\cdots & T_{N1}
\end{bmatrix})^{\perp}\\
=&(\mathrm{im}~T_{11}+ \mathrm{im}~T_{21}+\cdots \mathrm{im}~T_{N1})^{\perp}\\
=&\bigcap_{i=1}^N (\mathrm{im}~T_{i1})^{\perp}\\
=&\bigcap_{i=1}^N \mathrm{ker}~O_{Fi}\\
=&\mathrm{ker}\begin{bmatrix}
O_{F_1}\\
\vdots\\
O_{F_N}
\end{bmatrix}\\
=&0,
\end{array}
\end{equation}
where we have used the fact that the pair $(F,A)$ is observable.
This implies
\begin{equation}
\mathrm{rank}\begin{bmatrix}
T_{11} & T_{21} &\cdots & T_{N1}
\end{bmatrix}=n.
\end{equation}
Consequently, $\begin{bmatrix}
T_{11} & T_{21} &\cdots & T_{n1}
\end{bmatrix}$ has full row rank $n$, so we obtain:
\begin{equation}
(\frac{N}{\lambda _2}-\frac{N}{\lambda _2 +g_{min}})\begin{bmatrix}
T_{11} &\cdots & T_{N1}
\end{bmatrix} \begin{bmatrix}
T_{11} &\cdots & T_{N1}
\end{bmatrix}^T>0.
\end{equation}
We conclude that the left-hand side of (\ref{L1}) is positive definite, and consequently, for any choice of $g_i>0$, $i\in \mathcal{N}$, there exists a scalar $\epsilon>0$ such that inequality (\ref{L1}) holds.

\end{proof}

The following theorem now deals with the existence of a distributed observer of the form (\ref{obsl}) that achieves omniscience asymptotically with an a priori given error convergence rate. A condition for its existence is expressed in terms of solvability of a system of $N$ LMI's. Solutions to these LMI's yield the required gain matrices. Let $r_i>0$, $i\in \mathcal{N}$, be as in Lemma \ref{mirror}. Let $g_i>0$, $i\in \mathcal{N}$, and $\epsilon >0$ be such that (\ref{L1}) holds. Let $\gamma \in \mathbb{R}$. Finally, let $\alpha \geq 0$ be a desired error convergence rate. Recall the definitions \eqref{outgain1} and \eqref{Tis} for $S_i$ and $T_{is}$. We have the following:
\begin{theorem} \label{LMIresult}
There exist gain matrices $N_i$, $L_i$, $M_i$, $P_i$ and $Q_i$, $i\in \mathcal{N}$, such that the distributed observer (\ref{obsl}) achieves omniscience asymptotically and all solutions of the error system (\ref{sys_eg}) converge to zero with convergence rate at least $\alpha$ if there exist positive definite matrices $\mathcal{P}_{ie}\in \mathbb{R}^{(v_i-p_i)\times (v_i-p_i)},\mathcal{P} _{iu}\in \mathbb{R}^{(n-v_i)\times (n-v_i)}$, and a matrix $W_i\in \mathbb{R}^{(v_i-p_i)\times p_i}$ such that
\begin{equation} \label{Th1}
\begin{bmatrix}
\Phi _i+ \gamma  g_i I_{v_i -p_i}  &A_{i32}^T\mathcal{P}_{iu}\\
 \mathcal{P}_{iu}A_{i32} & A_{iu}^T \mathcal{P}_{iu} + \mathcal{P}_{iu}A_{iu} +2\alpha \mathcal{P}_{iu}
\end{bmatrix}-\gamma \epsilon I_{n-p_i}<0,~\forall i\in \mathcal{N},
\end{equation}
where $\Phi _i:= \mathcal{P}_{ie}A_{i22}+A_{i22}^T\mathcal{P}_{ie}-W_iE_iA_{i12}-A_{i12}^TE_i^TW_i^T+2\alpha \mathcal{P}_{ie}$. In that case, the gain matrices in the distributed observer (\ref{obsl}) can be taken as 
\begin{equation} \label{tlm}
K_i:=\begin{bmatrix}
E_i^{-1}\\
H_i\\
0
\end{bmatrix}D^{\dag} ,~ L_i:=\begin{bmatrix}
A_{i21}-H_iE_iA_{i11}\\
A_{i31}
\end{bmatrix}E_i^{-1}D_i^{\dag}+N_iS_i^T K_i
\end{equation}
\begin{equation} \label{tlm1}
M_i:=\begin{bmatrix}
\mathcal{P}_{ie}^{-1} &0\\
0 &\mathcal{P}_{iu}^{-1}
\end{bmatrix}T_{is}^T, ~N_i:=\begin{bmatrix}
A_{i22}-H_iE_iA_{i12} &0\\
A_{i32} &A_{iu}
\end{bmatrix},
\end{equation}
\begin{equation} \label{tlm11}
P_i := T_{is}, ~ Q_i := T_i K_i,
\end{equation}
where $H_{i} :=\mathcal{P}_{ie}^{-1}W_i$~, $i\in \mathcal{N}$.
\end{theorem}
\begin{proof}
By taking the gain matrices \eqref{tlm},  \eqref{tlm1} and  \eqref{tlm11}, the global error satisfies the differential equation \eqref{sys_eg}. According to Lemma \ref{invariance} we also have $e(t) \in \mathrm{im}~ T_s$ for all $t \in \mathbb{R}$. As a candidate Lyapunov function for the error system we take
\begin{equation}
V(e) = e^T \mathcal{P} e
\end{equation}
where $\mathcal{P}:=\mathrm{diag}\{\mathcal{P}_1,\cdots,\mathcal{P}_N\}$ and 
\[
\mathcal{P}_i:=T_i\begin{bmatrix}
I_{p_i} &0 &0\\
0 &\mathcal{P}_{ie} &0\\
0 &0 & \mathcal{P}_{iu}
\end{bmatrix}T_i^T.
\]
Clearly then $\mathcal{P} >0$.
The time-derivative of $V$ is
\begin{equation}\label{Ve}
\dot{V}(e)=e^T(\mathcal{P}T_sNT_s^T+T_sN^TT_s^T\mathcal{P}-\gamma \mathcal{P} T_s M(R\mathcal{L}\otimes I_n)-\gamma (\mathcal{L}^TR\otimes I_n)M^T T_s^T \mathcal{P})e
\end{equation}
with $T_s$, $M$ and $N$ the block diagonal versions of the $T_{is}$, $M_i$ and $N_i$  as defined by \eqref{Ts}, \eqref{M} and \eqref{N}.
By substituting $M_i:=\begin{bmatrix}
\mathcal{P}_{ie}^{-1} &0\\
0 &\mathcal{P}_{iu}^{-1}
\end{bmatrix}T_{is}^T$ into (\ref{Ve}), the time-derivative of $V$ becomes
\begin{equation} \label{derivative}
\dot{V}(e)=e^T\Lambda e,
\end{equation}
where we have defined 
\[
\Lambda := \mathcal{P}T_s N T_s^T+T_sN^TT_s^T \mathcal{P}-\gamma T_s T_s^T(R\mathcal{L}\otimes I_n)-\gamma (\mathcal{L}^TR\otimes I_n)T_s^T T_s.
\]
On the other hand, by combining  (\ref{Th1}) with (\ref{L1}) in Lemma \ref{key} it can be verified that 
\begin{equation}\label{Th1q}
\mathrm{diag}\{\mathcal{Q}_1,\cdots,\mathcal{Q}_N\}-T_s^T\gamma((R\mathcal{L}+\mathcal{L}^TR)\otimes I_n)T_s<0,
\end{equation}
where $\mathcal{Q}_i :=\begin{bmatrix}
\Phi _i &A_{i32}^TP_{iu}\\
\mathcal{P}_{iu}A_{i32} & \mathcal{P}_{iu}A_{iu}+A_{iu}^T\mathcal{P}_{iu}+2\alpha \mathcal{P}_{iu}
\end{bmatrix}$, $i\in \mathcal{N}$, and $\Phi _i$ as defined in the statement of the theorem.

Recall that we have defined $H_{i} :=\mathcal{P}_{ie}^{-1}W_i$. Hence $W_i = \mathcal{P}_{ie}H_i$. By substituting this into the expression for $\Phi_i$, we can check that
\[
\mathcal{Q}_i = T_{is}^T \mathcal{P}_i T_{is} N_i + N_i^T T_{is}^T \mathcal{P}_i T_{is} + 2 \alpha T_{is}^T\mathcal{P}_i T_{is}
\]
Substituting this into the inequality \eqref{Th1q}, using that $T_s^T T_s$ is the identity matrix, we get
\begin{equation*}
T_{s}^T(\mathcal{P} T_s N T_s^T+T_s N^T T_s^T \mathcal{P}-\gamma T_sT_s^T(R\mathcal{L}\otimes I_n)-\gamma (\mathcal{L}^TR\otimes I_n)T_sT_s^T+2\alpha \mathcal{P})T_{s}<0,
\end{equation*}
so, in other words, 
\begin{equation}   \label{eqv}
T_s^T (\Lambda + 2 \alpha \mathcal{P}) T_s < 0.
\end{equation}
By combining \eqref{derivative} and \eqref{eqv} we will now show that all global error trajectories converge to zero with convergence rate at least $\alpha$. Indeed let $e$ be such trajectory. By Lemma \ref{invariance} we have that $e$ can be represented as  $e= T_s z $ for some function $z$. Thus we get
\begin{equation*}
\begin{array}{ll}
\dot{V}(e)+ 2 \alpha V(e) & =  e^T \Lambda e + 2 \alpha e^T\mathcal{P} e \\
 & =z^T T_s^T( \Lambda + 2 \alpha \mathcal{P}) T_s z,
 \end{array}
\end{equation*}
and therefore $\dot{V}(e)(t)+ 2 \alpha V(e)(t) < 0$  whenever $e(t) \neq 0$.
Hence the distributed observer \eqref{obsl} achieves omniscience asymptotically and all solutions of the global error system converge to zero asymptotically with convergence rate at least $\alpha$.

\end{proof}

Using the previous lemmas and theorem, we are now able to formulate and prove our main result:
\begin{theorem} \label{th2}
Assume that $(C,A)$ is observable and that $\mathcal{G}$ is a strongly connected directed graph. Let $\alpha \geq 0$. Then there exists a distributed observer (\ref{obsl}) that achieves omniscience asymptotically while all solutions of the error system converge to zero with convergence rate at least $\alpha$. This distributed observer has state space dimension equal to $Nn - \sum_{i =1}^N p_i$ with $p_i = \mathrm{rank}~C_i$. Such observer is obtained as follows:
\begin{itemize}
\item[1] For each $i\in \mathcal{N}$, make a full rank factorization $C_i=D_iF_i$  where $D_i \in \mathbb{R}^{m_i\times p_i}$ and $F_i \in \mathbb{R}^{p_i\times n}$ have full column rank and row rank, respectively.
\item[2] For each $i\in \mathcal{N}$, choose an orthogonal matrix $T_i$ such that
\begin{equation} \label{dec}
T_i^TAT_i=\begin{bmatrix}
A_{i11} &A_{i12} &0\\
A_{i21} &A_{i22} &0\\
A_{i31} &A_{i32} &A_{iu}\\
\end{bmatrix},
~F_iT_i=\begin{bmatrix}
E_i &0 &0
\end{bmatrix}
\end{equation}
with the pair $(\begin{bmatrix} E_i & 0 \end{bmatrix}, \begin{bmatrix}
A_{i11} &A_{i12}\\
A_{i21} &A_{i22}
\end{bmatrix})$
observable and $E_i$ non-singular. Then $(E_iA_{i12},A_{i22})$ is also observable.
\item[3] Compute the positive row vector $r=\begin{bmatrix}
r_1,\cdots,r_N
\end{bmatrix} $ such that $r\mathcal{L}=0$ and $r\mathbf{1}_N=N$.
\item[4] Put $g_i=1$, $i\in \mathcal{N}$ and take $\epsilon >0$ such that (\ref{L1}) holds.
\item[5] Take $\gamma>0$ sufficiently large so that for all $i\in \mathcal{N}$ 
\begin{equation}\label{Thg1}
A_{iu}^T+A_{iu}-(\gamma \epsilon-2\alpha) I_{n-v_i}+\frac{1}{\gamma \epsilon-2\alpha}A_{i32}A_{i32}^T<0.
\end{equation}
\item[6] Choose $H_{i}$ such that all eigenvalues of $A_{i22}-H_{i}E_iA_{i12}$ lie in the region $\{s\in \mathbb{C}~ |~ \mathrm{Re}(s)<-\alpha\}$. 
\item[7] For all $i\in \mathcal{N}$, solve the Lyapunov equation 
\begin{equation}\label{Th2}
(A_{i22}-H_{i}E_iA_{i12}+\alpha I _{v_i-p_i})^T \mathcal{P}_{ie} + \mathcal{P}_{ie}(A_{i22}-H_{i}E_iA_{i12}+\alpha I _{v_i-p_i}) +(\gamma-2\alpha)  I_{v_i-p_i}=0
\end{equation}
to obtain $P_{ie}>0$.
\item[8] Define  
\begin{equation} 
K_i:=\begin{bmatrix}
E_i^{-1}\\
H_i\\
0
\end{bmatrix}D_i^{\dag} ,~S_i:=\begin{bmatrix}
0\\
I_{n-p_i}
\end{bmatrix}, T_{is} := T_i S_i,
\end{equation}
\begin{equation}
L_i:=\begin{bmatrix}
A_{i21}-H_iE_iA_{i11}\\
A_{i31}
\end{bmatrix}E_i^{-1}D_i^{\dag}+N_iS_i^T K_i,
\end{equation}
\begin{equation}
M_i:=\begin{bmatrix}
\mathcal{P}_{ie}^{-1} &0\\
0 & I_{n - v_i}
\end{bmatrix}T_{is}^T, ~N_i:=\begin{bmatrix}
A_{i22}-H_iE_iA_{i12} &0\\
A_{i32} &A_{iu}
\end{bmatrix},
\end{equation}
\begin{equation} 
P_i := T_{is}, ~ Q_i := T_i K_i.
\end{equation}
\end{itemize}
\end{theorem}
\begin{proof}
We choose $g_i=1$, $i\in \mathcal{N}$. Since the pair $(C,A)$ is observable and the graph $\mathcal{G}$ is a strongly connected directed graph, $\epsilon>0$ can be obtained by Lemma 5. 

Putting $\mathcal{P}_{iu}=I_{n-v_i}$, $i\in \mathcal{N}$, the inequality (\ref{Th1}) in Theorem \ref{LMIresult} becomes
\begin{equation}\label{Thg}
\begin{bmatrix}
\Phi _i+ \gamma  I_{v_i-pi} &A_{i32}^T\\
A_{i32} &A_{iu}^T+A_{iu}+2\alpha I_{n-v_i}
\end{bmatrix}-\gamma \epsilon I_n<0,~ \forall i\in \mathcal{N}.
\end{equation}
where $\Phi _i:= \mathcal{P}_{ie}A_{i22}+A_{i22}^T\mathcal{P}_{ie}-W_iE_iA_{i12}-A_{i12}^TE^TW_i^T+2\alpha \mathcal{P}_{ie}$.
By substituting (\ref{Th2}) and $W_i=\mathcal{P}_{ie}H_i$ into (\ref{Thg}), we have that the inequality (\ref{Thg}) holds if
\begin{equation}\label{Th3}
\begin{bmatrix}
-(\gamma \epsilon-2\alpha) I_{v_i-p_i} &A_{i32}^T\\
A_{i32} & A_{iu}^T + A_{iu}-(\gamma \epsilon -2\alpha) I_{n-v_i}
\end{bmatrix}<0,\forall i\in \mathcal{N}.
\end{equation}
By using the Schur complement lemma, (\ref{Th3}) is equivalent with
\begin{equation} \label{Th4}
A_{iu}+A_{iu}^T-(\gamma \epsilon -2\alpha) I_{n-v_i}+\frac{1}{\gamma \epsilon -2\alpha}A_{i32}A_{i32}^T<0,~ \forall i\in \mathcal{N}.
\end{equation}
As stated in step 5, inequality (\ref{Th4}) can be made to hold with sufficiently large $\gamma >0$.

Thus, we find that the parameters introduced in steps 4 to 7 guarantee that the inequality (\ref{Th1}) in Theorem \ref{LMIresult} holds. Hence, the distributed observer (\ref{obsl}) with gain matrices $N_i$, $L_i$, $M_i$, $P_i$  and $Q_i$ achieves omniscience asymptotically with convergence rate at least $\alpha$.

\end{proof}

\begin{remark}
\rm{For any given $\alpha \geq 0$, the coupling gain $\gamma >0$ can indeed be taken sufficiently large to guarantee that (\ref{Thg1}) holds. Since $(E_iA_{i12},A_{i22})$ is observable, for any $\alpha \geq 0$ the Lyapunov equation (\ref{Th2}) in step 7 can be made to have a positive definite solution by choosing the matrix $H_i$ as in step 6.}
\end{remark}

\begin{remark}
\rm{The design procedure in Theorem 7 gives one possible choice of solutions of the inequality (\ref{Th1}) in Theorem \ref{LMIresult}, which also means that under our standing assumptions that $(C,A)$ is observable and the graph $\mathcal{G}$ is strongly connected the inequality (\ref{Th1}) always has the required solutions. In fact, the inequalities (\ref{L1}) in Lemma \ref{key} and (\ref{Th1}) in Theorem \ref{LMIresult} are both LMI's, which can be solved numerically by using the LMI Toolbox or YALMIP in MATLAB directly.}
\end{remark}


\begin{remark}
\rm{In the special case that $C$ has full row rank $m$, all local output matrices $C_i$ have full row rank $m_i$ as well, so $p_i = m_i$ for all $i \in \mathcal{N}$. In this case our distributed observer has order $Nn - m$. In this case step 1 of our design procedure can be skipped since $F_i = C_i$ and $D_i = I_{m_i}$.}
\end{remark}

\begin{remark}
\rm {Another special case occurs if for some $i$ we have $v_i = p_i$, which means that $\mathrm{ker}~C_i$ coincides with the unobservable subspace of $(C_i,A)$. In this case, in the decomposition \eqref{dec} the second block column and row are void, so in particular $A_{i12}$, $A_{i22}$, $A_{i32}$ and $A_{i21}$ do not appear. Step 5 then reduces to  $A_{iu}+A_{iu}^T-(\gamma \epsilon-2\alpha) I_{n-p_i} < 0$, and steps 6 and 7 can be skipped. The local observer \eqref{obsl} at node $i$ is then given by
\begin{equation}\label{tlm3}
N_i:=A_{iu},~L_i:=A_{i31}E_i^{-1}D_i^{\dag},~M_i = T_{is}^T,~ P_i = T_{is},~ Q_i = T_i K_i,~  K_i:=\begin{bmatrix}
E_i^{-1}\\
0
\end{bmatrix}D_i^{\dag}.
\end{equation} }
\end{remark}

\section{Conclusions}
In this paper we have studied the problem of reduced order distributed observer design. We have shown that if the observed plant is observable and the network graph is strongly connected, then a distributed observer achieving omniscience exists with state space dimension equal to $Nn - \sum_{i = 1}^N p_i$, where $N$ is the number of network nodes, $n$ is the dimension of the plant state space 
and $p_i$ is the rank of the output matrix corresponding to the output received by node $i$. In fact, for any desired rate of error convergence a distributed observer of this order exists. As an intermediate result we have cast the distributed observer design problem in terms of feasiblity of LMI's, which is advantageous from a computational point of view. Under our standing assumptions these LMI's are always solvable. 

Whereas in the case of a single observer our reduced order is known to be the minimal state space dimension for a stable observer, it remains an open problem to determine the minimal order over all distributed observers with a given network graph. This is a left as a problem for future research.


%

\appendices
%



\ifCLASSOPTIONcaptionsoff
  \newpage
\fi



\bibliographystyle{IEEEtran}
%

%









\end{document}